\documentclass[leqno]{amsart}

\usepackage{epsf}
\usepackage{graphicx}
\usepackage{amsmath}
\usepackage[psamsfonts]{amssymb}
\usepackage[psamsfonts]{eucal}
\usepackage{amsthm}
\usepackage{textcomp}
\usepackage[scaled]{helvet}

\theoremstyle{plain}
\newtheorem{theorem}{Theorem}
\newtheorem{proposition}[theorem]{Proposition}

\newtheorem{corollary}[theorem]{Corollary}

\newtheorem{theoremX}{Main Theorem}

\theoremstyle{definition}

\newtheorem{remark}[theorem]{Remark}

\newcommand{\secref}[1]{Section~\ref{#1}}

\newcommand{\propref}[1]{Proposition~\ref{#1}}

\def\ad{{\rm ad}}

\def\inf{{\rm Inf}}

\def\im{{\rm Im}\,}

\def\L{\mathbb L}
\def\Z{\mathbb Z}
\def\Q{\mathbb Q}

\begin{document}

\title{Lawrence-Sullivan Models for the Interval}

\author{Paul-Eug\`ene Parent}
\address{D\'epartement de Math\'ematiques\\
        Universit\'e d'Ottawa\\
         585 King Edward\\
          Ottawa, ON\\
          K1N 6N5\\
         Canada}
\email{pparent@uottawa.ca}

\author{Daniel Tanr\'e}
\address{D\'epartement de Math\'ematiques\\
         UMR 8524\\
         Universit\'e des Sciences et Technologies de Lille\\
         59655 Villeneuve d'Ascq Cedex\\
         France}
\email{Daniel.Tanre@univ-lille1.fr}

\date{\today}

\begin{abstract}
Two constructions of a Lie model of the interval were performed by R. Lawrence and D. Sullivan. The first model uses an inductive process and the second one comes directly from solving a differential equation. They conjectured that these two models are the same. We prove this conjecture here.
\end{abstract}

\maketitle

This work is concerned with Lie models of the interval.
Throughout this paper we assume that the base field is the field of rational numbers. A {\it graded Lie algebra} consists of a $\Z$-graded vector space $L$, together with a bilinear product called the {\it Lie bracket} that we denote $[-,-]$, such that
$
[x,y]=-(-1)^{\vert x\vert\vert y\vert}[y,x]
$
and
$$
(-1)^{|x|\,|z]}[x,[y,z]]+(-1)^{|y|\,|x]}[y,[z,x]]+(-1)^{|z]\,|y]}[z,[x,y]]=0,$$
for all homogeneous $x,y,z\in L$, where $\vert \alpha\vert$ refers to the degree of a homogeneous element $\alpha\in L$.
If a graded Lie algebra is endowed with a derivation $\partial$ of degree $-1$ such that $\partial^2=0$, we call $(L,\partial)$ a {\it differential graded Lie algebra}, abbreviated {\it dgL}, and $\partial$ is its {\it differential}.

Let $V$ be a $\Z$-graded vector space, and let $TV$ denote the tensor algebra on $V$. When endowed with the commutator bracket, $TV$ becomes a graded Lie algebra. The {\it free Lie algebra generated by} $V$, denoted $\L V$, is the smallest sub Lie algebra of $TV$ containing $V$. An element in $\L V$ has {\it bracket length} $k$ if it is a linear combination of iterated brackets of $k$ elements of $V$, i.e., if it belongs to the intersection $\mathbb LV\cap T^kV$, where $T^kV$ denotes the subspace of $TV$ generated by the words of tensor length $k$.
The subspace of elements of bracket length $k$ is denoted $\mathbb L^kV$.  If $({\L}V,\partial)$ is a dgL,we denote by $\partial_k$ the derivation induced by the composition of $\partial$ with the projection
${\L}V\to{\L}^kV$.
We denote by $\widehat{\L}V$ the completed Lie algebra, whose elements are formal series of elements of $\L V$.

Let $X$ be a CW-complex with cells $(e_{\alpha})$ such that their closure $(\overline{e}_{\alpha})$ has the rational homology of a point. Denote by $V$ the rational vector space span by the desuspended cells, i.e., each cell $e_{\alpha}$ generates a component $\Q$ of $V$ in degree $|e_{\alpha}|-1$.
In an appendix to \cite{MR2308943}, D. Sullivan constructs a completed differential Lie algebra $(\widehat{\L}V,\partial)$, with $\partial=\partial_0+\partial_1+\partial_{\geq 2}$ such that $\partial_0\colon V\to V$ is the boundary operator of cells,
$\partial_1\colon V\to \widehat{\L}^2V$ comes from a cellular approximation of the diagonal and $\partial_{\geq 2}(V)\subset \widehat{\L}^{>2}V$. In the case of the interval $I$, with two 0-cells and one 1-cell, this model is of the shape
$(\widehat{\L}(a,b,x),\partial)$, with $|a|=|b|=-1$, $|x|=0$ and $\partial_0x=b-a$. More details are given in \secref{sec:sullivaninductive}.
We call this model \emph{the inductive model of the interval.}

In \cite{LrSd}, R. Lawrence and D. Sullivan prove the existence of a completed differential Lie algebra $(\widehat{\L}(a,b,x),\partial)$, such that
$\partial a=-(1/2)[a,a]$,
$\partial b= -(1/2)[b,b]$ and
$$\partial x=\ad_x(b)+\sum_{i=0}^{\infty}\frac{B_i}{i!}(\ad_x)^i(b-a),$$
where the $B_i$ are the Bernoulli numbers. This construction comes from an analysis of the flow generated by $x$ which moves from $a$ to $b$ in unit time, with $a$ and $b$ being flat. We call it \emph{the geometric model}.

In the two papers, \cite{MR2308943} and \cite{LrSd}, it is conjectured that these two models are the same. We prove this conjecture here.

\begin{theoremX}
The inductive and the geometric models of the interval are the same.
\end{theoremX}

The proof consists in two preliminary steps. In  \secref{sec:sullivaninductive}, we revisite the inductive construction taking into account the particular case of the interval. A second ingredient comes from the study of some derivations of ${\L}(x,\beta)$, $|x|=0$ and $|\beta|=-1$, done in \secref{sec:derivations}. Finally, \secref{sec:maintheorem} contains the proof of the conjecture using, among other tools,  the Euler formula which characterizes the Bernoulli numbers. As a bonus, our proof generates other relations between the Bernoulli numbers. In consideration of the litterature on the subject, they are certainly well-known but we do not have any reference for them.

Finally, one may observe as in \cite{LrSd} that this differential (completed) Lie algebra is similar to the Quillen's model, introduced in \cite{MR0258031} for the study of the rational homotopy types of CW-complexes with one 0-cell and no 1-cell, in contrast with the Sullivan approach (\cite{MR0646078})  which authorizes nilpotent spaces of finite type. Through the dictionary between infinity cocommutative coassociative coalgebra structure and differential of a free Lie algebra, this model is an explicit construction of the infinity cocommutative coassociative coalgebra structure on the rational chains of the interval.

\section{Sullivan's inductive construction}\label{sec:sullivaninductive}

 In this section, we adapt Sullivan's proof of \cite{MR2308943} to the particular case of the interval, and prove the next property.

\begin{proposition}\label{prop:inductive}
The Sullivan's inductive model $(\widehat{\L}(a,b,x),\partial)$ of the interval admits a differential  of the form
$
\partial a\,=\,-\frac{1}{2}[a,a]$, $\partial b\,=\,-\frac{1}{2}[b,b]
$
and
$$
\partial x\,=\,ad_x(b)\,+\,\sum_{i=0}^{\infty}\,\frac{\lambda_i}{i!}ad_x^i(\beta),
$$
with $\beta=b-a$, $\lambda_i\in\mathbb Q$ and $\lambda_{2k+1}=0$, $k\geq1$.
\end{proposition}

\begin{proof}
 We proceed by induction, supposing that a derivation $\partial_{\leq n}$ on $\mathbb L(a,b,x)$ has been constructed such that
$$
\partial_{\leq n} a\,=\,-\frac{1}{2}[a,a]\,,\;
\partial_{\leq n} b\,=\,-\frac{1}{2}[b,b]\,,\;
\partial_{\leq n} x\,=\,ad_x(b)\,+\,\sum_{i=0}^{n}\frac{\lambda_i}{i!}ad_x^i(\beta),
$$
 with $\lambda_i\in\mathbb Q$  satisfying
$$
\im \partial_{\leq n}^2\subset\mathbb L^{\geq n+2}(x,\beta), \text{ for } n\geq 1,
$$
and
$$
\im \partial_{n}\subset \L(x,\beta), \text{ for } n\geq 2.$$
The induction starts at $n=1$. We first verify that case. Consider the derivations $\partial_0$ and $\partial_1$ on $\L(a,b,x)$, given by
$
\partial_0a=\partial_0b=0$, $\partial_0x=\beta,
$
and
$$
\partial_1a=-\frac{1}{2}[a,a]\,,\;\partial_1b=-\frac{1}{2}[b,b]\,,\;\partial_1x=\frac{1}{2}ad_x(b+a).
$$
Clearly $\partial_0^2=\partial_0\partial_1+\partial_1\partial_0=0$
and we proceed with the computation of $\partial_1^2$. The Jacobi identity implies that triple brackets of the form $[\alpha,[\alpha,\alpha]]$ are all zero. Hence $\partial_1^2a=\partial_1^2b=0$. In contrast, $\partial_1^2x\neq0$. A computation using the Jacobi identity shows that
$$\partial_1^2x\,=\,\frac{1}{2}\,\partial_1(ad_x(b+a))\,=\,-\frac{1}{8}ad_x[\beta,\beta]\in \L^3(x,\beta).$$
Suppose now that $\partial_{\leq n}$ has been constructed as before.
The set of derivations on a Lie algebra carries a natural Lie structure. Hence, we must have
\begin{eqnarray}\label{triple}
[\partial_{\leq n},[\partial_{\leq n},\partial_{\leq n}]]=0.
\end{eqnarray}
Moreover, since $\vert\partial_{\leq n}\vert=-1$, we have $[\partial_{\leq n},\partial_{\leq n}]=2\,\partial_{\leq n}^2$ and  $[\partial_{\leq n},\partial_{\leq n}](x)\in\mathbb L^{\geq n+2}(x,\beta)$ by the induction hypothesis. Hence the component in $\mathbb L^{n+2}(x,\beta)$ of the triple bracket (\ref{triple}) evaluated at $x$  is given by
$$
0=[\partial_0,\sum_{i+j=n+1}\partial_i\circ \partial_j](x)=\partial_0(\sum_{i+j=n+1}\partial_i\circ\partial_j)(x),
$$
The element $\displaystyle \sum_{i+j=n+1}\Big(\partial_i\circ \partial_j\Big)(x)\in\mathbb L^{n+2}(x,\beta)$ being a $\partial_0$-cycle of total degree $-2$ and the dgL $(\mathbb L(x,\beta),\partial_0)$ being acyclic, one can find an element $\gamma\in\mathbb L^{n+2}(x,\beta)$ of total degree $-1$ such that
$$
\partial_0\gamma=\sum_{i+j=n+1}\Big(\partial_i\circ\partial_j\Big)(x).
$$
The subspace of $\mathbb L^{n+2}(x,\beta)$ in total degree $-1$ being generated by $ad_x^{n+1}(\beta)$, $\gamma$ must be some (rational) multiple $\eta$ of that element. Hence we can extend $\partial_{\leq n}(x)$ with $$\partial_{n+1}(x)=\frac{\lambda_{n+1}}{(n+1)!}\ad_x^{n+1}(\beta)$$ by choosing $\lambda_{n+1}=-(n+1)!\cdot\eta$. By construction one has the inclusion
$$
\im\partial_{\leq n+1}^2\subset\mathbb L^{\geq n+3}(x,\beta),\;{\text{ and }}\,
\im\partial_{n+1}\subset \L(x,\beta).
$$
This completes the inductive step. Let us show that one can choose $\lambda_{2k+1}=0$ if $k\geq 1$.
First, notice that the restriction of $\partial_1$ to $\mathbb L(x,\beta)$ satisfies
$$
\partial_1\vert_{\mathbb L(x,\beta)}=-\frac{1}{2}ad_{a+b}.
$$
By definition, it is true on $x$ while on $\beta$ we have
$$
\partial_1\beta=-\frac{1}{2}[b,b]+\frac{1}{2}[a,a]=-\frac{1}{2}[b+a,b-a]=-\frac{1}{2}ad_{a+b}\beta.
$$
Hence for all $k$, $1<k$, we have
$$
\partial_1\circ\partial_k=-\frac{1}{2}ad_{a+b}\circ\partial_k,
$$
since by hypothesis $\partial_k(\mathbb L(a,b,x))\subset\mathbb L(x,\beta)$.
Moreover, we have
$
\partial_k\circ\partial_1(a)=\partial_k\circ\partial_1(b)=0
$
and
$
\partial_k\circ\partial_1(x)=-\frac{1}{2}\partial_kad_{a+b}(x)=\frac{1}{2}ad_{a+b}\partial_k(x)
$,
i.e.,
$$
\partial_k\circ\partial_1=\frac{1}{2}ad_{a+b}\circ \partial_k.
$$
Suppose that $\lambda_{2k+1}=0$ for $1<2k+1<n$ with $n$ even. We can identify the $(n+2)$-bracket length of $\partial_{\leq n}^2(x)$ with
$(\partial_1\circ\partial_n+\partial_n\circ\partial_1)(x)$, which is trivial as we have seen.
Hence,  we can extend $\partial_{\leq n}$ to $\partial_{\leq n+1}$ by choosing $\lambda_{n+1}=0$. This ends the proof.
\end{proof}

\section{Derivations of $\mathbb L(x,\beta)$}\label{sec:derivations}

Consider a graded Lie algebra $L=\mathbb L(a,b,x)$ generated by two elements $a$ and $b$ of degree $-1$, and one element $x$ of degree $0$.
Let $\beta=b-a$ and $\gamma\in\L(x,\beta)$ be an element of degree -1.
We set
 $$\mu_{n,k}(\gamma)=\ad_x^{n-2k}([\ad_x^k(\gamma),\ad_x^k(\gamma)])\in\L^{n+2}(x, \beta).$$
The elements $\mu_{n,k}(\beta)$ of $L$ play a crucial role in the inductive model.
To give a more explicit expression of this model, we study the composition of some adjoint representations in ${\L}(x,\beta)$. Let $v_{n,k}$ be the rational numbers defined by
\begin{eqnarray*}
v_{n,k}&=&v_{n-1,k}-v_{n-2,k-1},\\
v_{0,0}&=&1,\;v_{1,0}=1/2,\\
v_{i,j}&=&0\text{ if }j<0 \text{ or if }j>[i/2].
\end{eqnarray*}

\begin{proposition}\label{prop:lacle}
Let $\gamma\in\L(x,\beta)$ be an element of degree -1.
For any $p\geq 0$ and $q\geq 0$, we have
$$\ad_{\ad_x^p(\gamma)}\circ \ad_x^q(\gamma)=\sum_kv_{|p-q|,k}\,\mu_{p+q,\inf(p,q)+k}(\gamma).$$
\end{proposition}

\begin{proof}
This property is clearly satisfied for $p=0$ and $q=0$. Suppose it is true for $p=0$ and $0\leq q\leq n-1$. The Jacobi identity and the induction hypothesis imply:
\begin{eqnarray*}
\ad_{\gamma}\circ \ad_x^n(\gamma)&=&
\ad_x([\ad_x^{n-1}(\gamma), \gamma])+[\ad_x^{n-1}(\gamma),[\gamma,x]]\\
&=& \ad_x\circ \ad_{\gamma}\circ \ad_x^{n-1}(\gamma)-\ad_{[x, \gamma]}\circ \ad_x^{n-2}([x, \gamma])\\
&=& \ad_x(\sum_kv_{n-1,k}\,\mu_{n-1,k}(\gamma))-\sum_kv_{n-2,k}\,\mu_{n-2,k}([x, \gamma])\\
&=&\sum_k\left(v_{n-1,k}-v_{n-2,k-1}\right)\,\mu_{n,k}(\gamma).
\end{eqnarray*}
Thus the formula is proved for $p=0$. We distinguish now the  two cases, $p\geq q$ and $q\geq p$. Let $i\geq 0$ and $j\geq 0$.
\begin{itemize}
\item If $p=j$, $q=i+j$, then from $$\ad_{\ad_x^j(\gamma)}\circ \ad_x^{j+i}(\gamma)= \ad_{\ad_x^j(\gamma)}\circ \ad_x^i(\ad_x^j(\gamma))
$$ and from our first step applied to $\ad_x^j(\gamma)$, we get
$$\ad_{\ad_x^j(\gamma)}\circ \ad_x^{j+i}(\gamma)=\sum_{k}v_{i,k}\,\mu_{i+2j,j+k}(\gamma).$$
\item If $p=i+j$ and $q=j$, then using Jacobi identity and the first step, we have
$$
\ad_{\ad_x^i(\gamma)}(\gamma)= [\ad_x^i(\gamma), \gamma]=\ad_{\gamma}\circ \ad_x^i(\gamma)=\sum_k v_{n,k}\,\mu_{n,k}(\gamma).$$
Now replacing $\gamma $ by $\ad_x^i(\gamma)$ to get
$$\ad_{\ad_x^{i+j}(\gamma)}\circ \ad_x^i(\gamma)=\sum_{k} v_{j,k}\,\mu_{j+2i,k+i}(\gamma).$$
\end{itemize}
\end{proof}

\begin{corollary}\label{cor:independance}
The elements $\mu_{n,k}(\beta)\in \L(x,\beta)$ are linearly independent.
\end{corollary}

\begin{proof}
Recall first a well-known property concerning free Lie algebras (see \cite[Proposition VI.2.(7) Page 139]{MR764769} for instance).
Let $V$ and $W$ be rational vector spaces. Then the kernel of the canonical projection $\L(V\oplus W)\to \L(V)$ is the free Lie algebra on $T(V)\otimes W$, the canonical injection $\L(T(V)\otimes W)\to \L(V\oplus W)$ corresponding to adjunctions. More precisely, in our case, the kernel of the projection $\L(x,\beta)\to \L(x)$ is $\L(T(x)\otimes \beta)$ and the canonical inclusion
$j\colon \L(T(x)\otimes \beta)\to \L(x,\beta)$ is defined by $j(x^n\otimes \beta)=\ad_x^n(\beta)$.
We repeat this process: the kernel of the projection
$\L(T(x)\otimes\beta)\to \L(\beta)$
is $\L(T(\beta)\otimes T^+(x)\otimes\beta)$.

Let $n>0$ be fixed and $L \langle n \rangle$ be the vector subspace of $\L(x,\beta)$ formed of brackets with exactly $n$ letters $x$ and twice the letter $\beta$.
With the identification coming from the canonical inclusions,  $L \langle n \rangle$ is a subspace of $\L(T(\beta)\otimes T^+(x)\otimes\beta)$ spanned by the generator $\beta\otimes x^n\otimes \beta$ and the brackets $[x^i\otimes\beta,x^j\otimes \beta]$, with $i+j=n$.
If we impose $i\geq j$, these elements are a basis of  $L \langle n \rangle$.
Suppose that $n=2p$ for sake of simplicity, the argument being similar for $n$ odd.
The previous considerations show that $\L\langle n\rangle$ is of dimension $p+1$.
Now \propref{prop:lacle} implies that the $(p+1)$ elements  $\mu_{n,k}(\beta)$ span $\L\langle n\rangle$. Thus, they are a basis of $\L\langle n\rangle$.
\end{proof}

We denote by
$\theta_n$ the derivation of
${\L}(x,\beta)$, defined by
$
\theta_n(x)=ad_x^n(\beta)$ and $\theta_n(\beta)=0
$, for any $n\geq 0$.

\begin{corollary}\label{teta}
The image of the composition $\theta_p\circ\theta_q$ is contained in the linear span of the $\mu_{p+q-1,k}(\beta)$. More precisely,
if $p\geq q$, then we have
$$
\theta_p\circ\theta_q(x)=
\sum_k \left(\sum_{i=0}^{q-1} v_{p-i,k-i}\right)\,\mu_{p+q-1,k}(\beta),
$$
and when $p<q$, the formula becomes
$$
\theta_p\circ \theta_q(x)=\sum_k\left(\sum_{i=0}^{q-p-1}v_{q-p-1-i,k-p}+\sum_{i=1}^p v_{i,k-p+i}\right)\,\mu_{p+q-1,k}(\beta).
$$
\end{corollary}

\proof Clearly $\theta_p\circ\theta_0\equiv 0$. So we assume $p\geq0$ and $q\geq1$.
\begin{eqnarray*}
\theta_p\circ\theta_q(x)&=&[\theta_p(x),\theta_{q-1}(x)]+[x,\theta_p\circ\theta_{q-1}(x)],\\&=&
\ad_{\ad^p_x(\beta)}\circ\ad^{q-1}_x(\beta)+\ad_x\circ \theta_p\circ\theta_{q-1}(x).
\end{eqnarray*}
Hence, with an iteration on the second term and \propref{prop:lacle},  one gets
\begin{eqnarray*}
\theta_p\circ\theta_q(x)&=&\sum_{i=0}^{q-1}\ad_x^{i}\circ \ad_{\ad_x^p(\beta)}\circ \ad_x^{q-i-1}(\beta)\\
&=&
\sum_{i=0}^{q-1}\ad_x^{i}\circ \sum_tv_{|p-q+i+1|,t}\,\mu_{p+q-i-1,\inf(p,q-i-1)+t}(\beta)\\
&=&
\sum_{i=0}^{q-1} \sum_tv_{|p-q+i+1|,t}\,\mu_{p+q-1,\inf(p,q-i-1)+t}(\beta).
\end{eqnarray*}
If $p\geq q$, this formula simplifies in
\begin{eqnarray*}
\theta_p\circ\theta_q(x)&=&
\sum_{i=0}^{q-1} \sum_tv_{p-q+i+1,t}\,\mu_{p+q-1,q-i-1+t}(\beta)\\
&=&
\sum_k\sum_{j=0}^{q-1}v_{p-j,k-j}\,\mu_{p+q-1,k}(\beta),
\end{eqnarray*}
with $k=q+t-i-1$ and $j=q-1-i$.
If $p<q$, we have to cut the formula in two parts
\begin{eqnarray*}
\theta_p\circ\theta_q(x)&=&
\sum_t\sum_{i=0}^{q-p-1}v_{q-p-1-i,t}\,\mu_{p+q-1,p+t}(\beta)\\&&\hskip 2cm
+\sum_t\sum_{i=q-p}^{q-1}v_{p-q+1+i,t}\,\mu_{p+q-1,q-i-1+t}(\beta).
\end{eqnarray*}
The result follows by a simple but tedious re-indexing, as we did in the first case.
\qed

\section{Proof of the Main Theorem}\label{sec:maintheorem}
There is no universally accepted  convention for the Bernoulli numbers. Here we choose the following one:
$$\frac{z}{e^z-1}=\sum_{k=0}^{\infty}\frac{B_k}{k!}\,z^k=1-\frac{z}{2}+\sum_{k=1}^{\infty}\frac{B_{2k}}{(2k)!}\,z^{2k}.$$
All the Bernoulli numbers are rational, with
$B_1=-(1/2)$, $B_{2k+1}=0$ if $k\geq 1$,
$$B_0=1,\;
B_2=\frac{1}{6},\;
B_4=-\frac{1}{30},\;
B_6=\frac{1}{42},\;
B_8=-\frac{1}{30},\;
B_{10}=\frac{5}{66},\;
B_{12}=-\frac{691}{2730}.$$
Bernoulli numbers verify several induction formulae. Amongst them we recall the Euler formula, i.e.,
$$-n\,B_n=\sum_{k=1}^n \left(\begin{array}{c}n\\
k\end{array}\right)B_k\,B_{n-k}+n\,B_{n-1}.$$
When $n$ is even with $n> 2$, the formula reduces to
\begin{equation}
-\frac{(n+1)\,B_n}{n!}=\sum_{k=2}^{n-2}\frac{B_k}{k!}\,\frac{B_{n-k}}{(n-k)!}.
\end{equation}

\begin{theorem}
The coefficients of the differential in \propref{prop:inductive} are given by the Bernouilli numbers, i.e.,
$$
\lambda_i=B_i,
$$
for all $i\geq0$.
\end{theorem}

\proof
We check easily that $\lambda_0=B_0=1$ and $\lambda_1=B_1=-\frac{1}{2}$. Moreover, the construction guarantees that $\lambda_{2k+1}=B_{2k+1}=0$ for $k\geq1$.
Recall that within the proof of Proposition 1, we have shown that $\partial_1^2(x)=-\frac{1}{8}ad_x[\beta,\beta]$. A simple computation shows that $\partial_0([x,[x,\beta]])=\frac{3}{2}ad_x[\beta,\beta]$. That implies that we can choose $\lambda_2=B_2=\frac{1}{6}$.

Now let us assume $n>2$ and even. Observe from Corollary \ref{teta} that, for $p,q\geq2$ such that $p+q=n$, the rational coefficient of $\mu_{n-1,0}(\beta)$ in the expression of  $\theta_p\circ\theta_q(x)$ is given by
$$
v_{p,0}=\frac{1}{2}.
$$
Hence, on one hand, the coefficient of $\mu_{n-1,0}(\beta)$ in the equation
$$
\sum_{k=2}^{n-2}\partial_k\circ \partial_{n-k}(x)\,=\,\sum_{k=2}^{n-2}\frac{\lambda_k}{k!}\frac{\lambda_{n-k}}{(n-k)!}\theta_k\circ\theta_{n-k}(x)
$$
is
$$
\frac{1}{2}\sum_{k=2}^{n-2}\frac{\lambda_k}{k!}\frac{\lambda_{n-k}}{(n-k)!}.
$$
On the other hand, one has by Corollary \ref{teta} that the rational coefficient of $\mu_{n-1,0}(\beta)$ in the expression of  $\theta_0\circ\theta_n(x)$ is
$$
1+\sum_{k=1}^{n-1}v_{k,0}=1+\frac{n-1}{2}=\frac{n+1}{2}.
$$
Since by construction we have
$$
\partial_0\circ \partial_n\,+\,\sum_{i=2}^{n-2}\partial_i\circ \partial_{n-i}=0,
$$
the following relation must be satisfied, i.e.,
$$
\sum_{k=2}^{n-2}\frac{\lambda_k}{k!}\frac{\lambda_{n-k}}{(n-k)!}=-(n+1)\frac{\lambda_n}{n!},
$$
which is no other than the Euler equation characterizing the Bernouilli numbers. The result follows.

\qed

\section{Other Euler type relations between Bernouilli numbers}\label{sec:bernoullirelations}

The proof generates other relations amongst the Bernouilli numbers. Indeed, when $n$ is even, we can deduce such relations by projecting
$$
\partial_0\circ \partial_n\,+\,\sum_{i=2}^{n-2}\partial_i\circ \partial_{n-i}=0,
$$
onto any $\mu_{n-1,k}(\beta)$.

\begin{proposition}
For any fixed $k$, we have
\begin{eqnarray*}
-\frac{B_n}{n!}\left(\sum_{l=0}^{n-1}v_{l,k}\right)&=&
\sum_{\begin{array}{c}2\leq i<\frac{n}{2}\\k<i\end{array}}\frac{B_i}{i!}\frac{B_{n-i}}{(n-i)!}\left(\sum_{l=0}^iv_{l,k-i+l}\right)\\
&&\hskip 1cm +\sum_{\begin{array}{c}2\leq i<\frac{n}{2}\\k\geq i\end{array}}\frac{B_i}{i!}\frac{B_{n-i}}{(n-i)!}\left(\sum_{l=0}^{n-2i-1}v_{l,k-i}\right)\\
&&\hskip 1cm +\sum_{i=\frac{n}{2}}^{n-2}\frac{B_i}{i!}\frac{B_{n-i}}{(n-i)!}\left(\sum_{l=0}^{n-i-1}v_{i-l,k-l}\right).
\end{eqnarray*}
\end{proposition}

\begin{proof}
We project
$
\partial_0\circ \partial_n\,+\,\sum_{i=2}^{n-2}\partial_i\circ \partial_{n-i}=0,
$
onto $\mu_{n-1,k}(\beta)$.
On one hand, the coefficient of $\mu_{n-1,k}(\beta)$ in the expression of $\partial_0\circ \partial_n(x)$ is
$$
\frac{B_n}{n!}\left(\sum_{l=0}^{n-1}v_{l,k}\right).
$$
While on the other hand, when $p+q=n$ and $p<q$, the coefficient of $\mu_{n-1,k}(\beta)$ in the expression of $\partial_p\circ \partial_q(x)$ is given by
$$
\frac{B_p}{p!}\frac{B_q}{q!}\left\{\begin{array}{cc}\displaystyle
\left(\sum_{l=0}^pv_{l,k-p+l}\right),&k<p\\
\displaystyle\left(\sum_{l=0}^{q-p-1}v_{l,k-p}\right),&k\geq p
\end{array}\right.,
$$
and when $p\geq q$ it is given by
$$
\frac{B_p}{p!}\frac{B_q}{q!}\left(\sum_{l=0}^{q-1}v_{p-l,k-l}\right).
$$
\end{proof}

\begin{remark}
The previous relations cannot be directly reduced to Euler equation. For instance, when $n=8$ and $k=2$, one gets the relation
$$
\frac{9}{2}\frac{B_2}{2!}\frac{B_6}{6!}=-15\frac{B_8}{8!}
$$
which differs from the Euler equation simply by the fact that the coefficient of $\displaystyle\frac{B_4^2}{(4!)^2}$ is zero. If one considers the case $n=10$ and $k=2$, we get the relation
$$
\frac{5}{2}\frac{B_4}{4!}\frac{B_6}{6!}+10\frac{B_2}{2!}\frac{B_8}{8!}=-\frac{77}{2}\frac{B_{10}}{10!},
$$
in which no terms are missing but still differs from Euler's relation because the coefficients on the left hand side are not equal.
\end{remark}

Finally, we observe that the numbers $v_{n,k}$ can be explicitely determined.

\begin{proposition}
The sequence $v_{n,k}$ satisfy the following properties:
$$
2\,v_{n,k}=(-1)^k\left(\!\!\left(
\begin{array}{c}n-k\\
k
\end{array}\right)+
\left(
\begin{array}{c}n-k-1\\k-1\end{array}\right)\!\!\right),
\text{ and }\; \sum_{k=0}^nv_{n+k,k}=0.
$$
\end{proposition}

\begin{proof}
Let $f_0(n)=n$. We set
$$
f_{k+1}(n)=\sum_{i=1}^nf_{k}(n).$$
Recall from  \cite[Page 134]{MR1504010} that
\begin{equation}\label{equa:occagne}f_{k+1}(n)=
\left(
\begin{array}{c}n+k-1\\
k
\end{array}\right)
f_0(1)+\cdots +
\left(
\begin{array}{c}k\\
k
\end{array}\right)
f_0(n)=
\left(
\begin{array}{c}n+k+1\\
k+2
\end{array}\right).
\end{equation}
In the above expression of $v_{n,k}$, the sign is clear thus we have only to study the absolute value.
We introduce $\sigma_{n,k}=\sum_{i=0}^nv_{i,k}$. If we add, from 0 to $n$, the defining relation
$v_{n,k}=v_{n-1,k}-v_{n-2,k-1}$, we get
\begin{equation}\label{equa:uneautre}v_{n,k}=-\sum_{i=0}^{n-2}v_{i,k-1}=-\sigma_{n-2,k-1},
\end{equation}
and, by adding these relations from 0 to $k$, we get
\begin{equation}\label{equa:sigma}\sigma_{n,k}=-\sum_{i=0}^{n-2}\sigma_{i,k-1}.
\end{equation}
We observe that $2|\sigma_{n,0}|=n+2=f_0(n)+2$ and
\begin{eqnarray*}
2\,|\sigma_{n,1}|&=&\sum_{i=0}^{n-2}|\sigma_{i,0}|
=\sum_{i=0}^{n-2}(f_0(i)+2)\\
&=&f_1(n-2)+2\,f_0(n-1).
\end{eqnarray*}
More generally, an induction on k  using formula~(\ref{equa:sigma}) gives
$$2\,|\sigma_{n,k}|=f_k(n-2k)+2\,f_{k-1}(n-2k+1).$$
Formulae (\ref{equa:occagne}), (\ref{equa:uneautre})  and  basic properties of binomial coefficients give the result.

\end{proof}
For convenience, we supply the first values of the $v_{n,k}$.

\medskip\centerline{\large
\begin{tabular}{|c||r|r|r|r|}
\hline
$n$&$v_{n,0}$&
$v_{n,1}$&
$v_{n,2}$&$v_{n,3}$\\
\hline\hline
0&1&&&\\\hline
1&1/2&&&\\\hline
2&1/2&-1&&\\\hline
3&1/2&-3/2&&\\\hline
4&1/2&-2&1&\\\hline
5&1/2&-5/2&5/2&\\\hline
6&1/2&-3&9/2&-1\\\hline
7&1/2&-7/2&7&-7/2\\
\hline
\end{tabular}}

\bibliographystyle{amsplain}
\bibliography{BernoulliCMH}

\providecommand{\bysame}{\leavevmode\hbox to3em{\hrulefill}\thinspace}
\providecommand{\MR}{\relax\ifhmode\unskip\space\fi MR }
\providecommand{\MRhref}[2]{%
  \href{http://www.ams.org/mathscinet-getitem?mr=#1}{#2}
}
\providecommand{\href}[2]{#2}
\begin{thebibliography}{1}

\bibitem{MR1504010}
Maurice D'Ocagne, \emph{Sur une source d'identit\'es}, Bull. Soc. Math. France
  \textbf{15} (1887), 133--143. \MR{MR1504010}

\bibitem{LrSd}
Ruth Laurence and Dennis Sullivan, \emph{A free differential lie algebra for
  the interval}, Arxiv. math.AT/0610949, 2006.

\bibitem{MR0258031}
Daniel Quillen, \emph{Rational homotopy theory}, Ann. of Math. (2) \textbf{90}
  (1969), 205--295. \MR{MR0258031 (41 \#2678)}

\bibitem{MR0646078}
Dennis Sullivan, \emph{Infinitesimal computations in topology}, Inst. Hautes
  \'Etudes Sci. Publ. Math. (1977), no.~47, 269--331 (1978). \MR{MR0646078 (58
  \#31119)}

\bibitem{MR764769}
Daniel Tanr{\'e}, \emph{Homotopie rationnelle: mod\`eles de {C}hen, {Q}uillen,
  {S}ullivan}, Lecture Notes in Mathematics, vol. 1025, Springer-Verlag,
  Berlin, 1983. \MR{MR764769 (86b:55010)}

\bibitem{MR2308943}
Thomas Tradler and Mahmoud Zeinalian, \emph{Infinity structure of {P}oincar\'e
  duality spaces}, Algebr. Geom. Topol. \textbf{7} (2007), 233--260, Appendix A
  by Dennis Sullivan. \MR{MR2308943 (2009d:57047)}

\end{thebibliography}

\end{document}